\documentclass{amsart}
\usepackage{amssymb, amsmath, latexsym, dsfont,tikz}
\usepackage{multirow}
\usepackage{setspace}
\usepackage{enumerate}
\usepackage{diagbox, faktor}
\usepackage{stmaryrd}
\usepackage{makecell}
\usepackage{longtable, graphicx}

\renewcommand{\baselinestretch}{\baselinestretch}
\renewcommand{\baselinestretch}{1.1}
\numberwithin{equation}{section}

\newtheorem{thm}{Theorem}[section]
\newtheorem{lem}[thm]{Lemma}

\newtheorem{prop}[thm]{Proposition}

\theoremstyle{definition}

\theoremstyle{remark}
\newtheorem{rmk}[thm]{Remark}

\numberwithin{equation}{section}

\newcommand{\ra}{{\ \longrightarrow \ }}
\newcommand{\nra}{{\ \longarrownot\longrightarrow \ }}

\newcommand{\ord}{\text{ord}}
\newcommand{\n}{{\mathbb N}}
\newcommand{\z}{{\mathbb Z}}
\newcommand{\q}{{\mathbb Q}}

\newcommand{\tail}{a_1\alpha_1^2+a_2\alpha_2^2+a_3\alpha_3^2}
\newcommand{\be}{{\mathbf e}}
\newcommand{\bv}{{\mathbf v}}
\newcommand{\bw}{{\mathbf w}}
\newcommand{\bx}{{\mathbf x}}
\newcommand{\by}{{\mathbf y}}
\newcommand{\bz}{{\mathbf z}}

\begin{document}
%%%%%%%%%%%%%%%%%%%%%%%%%%%%%%%%%%%%%%%%%%%%%%%%%%%%%%%%%%%%%%%%%%%%%%%%%%%%%
%%%%%%%%%%%%%%%%%%%%%%%%%%%%%%%%%%%%%%%%%%%%%%%%%%%%%%%%%%%%%%%%%%%%%%%%%%%%%
\title[]{Regular ternary sums of generalized polygonal numbers}

\author[] {Mingyu Kim}

\address{Department of Mathematics Education, Pusan National University, Busan 46241, Korea}
\email{mingyukim@pusan.ac.kr}
\thanks{This work was suppoted by a 2-Year Research Grant of Pusan National University.}


%%%%%%%%%%%%%%%%%%%%%%%%%%%%%%%%%%%%%%%%%%%%%%%%%%%%%%%%%%%%%%%%%%%%%%%%%%%%%%%%%%%%%%%%%%

\begin{abstract}
In this article, we provide an explicit constant $C$ such that there is no regular ternary sum of generalized $m$-gonal numbers for any integer $m$ greater than $C$.
\end{abstract}

%%%%%%%%%%%%%%%%%%%%%%%%%%%%%%%%%%%%%%%%%%%%%%%%%%%%%%%%%%%%%%%%%%%%%%%%%%%%%%%%%%%%%%%%%%
\maketitle

%\nocite{*}

%%%%%%%%%%%%%%%%%%%%%%%%%%%%%%%%%%%%%%%%%%%%%%%%%%%%%%%%%%%%%%%%%%%%%%%%%%%%%%%%%%%%%%%%%%
%%%%%%%%%%%%%%%%%%%%%%%%%%%%%%%%%%%%%%%%%%%%%%%%%%%%%%%%%%%%%%%%%%%%%%%%%%%%%%%%%%%%%%%%%%
%%%%%%%%%%%%%%%%%%%%%%%%%%%%%%%%%%%%%%%%%%%%%%%%%%%%%%%%%%%%%%%%%%%%%%%%%%%%%%%%%%%%%%%%%%
%%%%%%%%%%%%%%%%%%%%%%%%%%%%%%%%%%%%%%%%%%%%%%%%%%%%%%%%%%%%%%%%%%%%%%%%%%%%%%%%%%%%%%%%%%
\section{Introduction}
%%%%%%%%%%%%%%%%%%%%%%%%%%%%%%%%%%%%%%%%%%%%%%%%%%%%%%%%%%%%%%%%%%%%%%%%%%%%%%%%%%%%%%%%%%
%%%%%%%%%%%%%%%%%%%%%%%%%%%%%%%%%%%%%%%%%%%%%%%%%%%%%%%%%%%%%%%%%%%%%%%%%%%%%%%%%%%%%%%%%%
%%%%%%%%%%%%%%%%%%%%%%%%%%%%%%%%%%%%%%%%%%%%%%%%%%%%%%%%%%%%%%%%%%%%%%%%%%%%%%%%%%%%%%%%%%
%%%%%%%%%%%%%%%%%%%%%%%%%%%%%%%%%%%%%%%%%%%%%%%%%%%%%%%%%%%%%%%%%%%%%%%%%%%%%%%%%%%%%%%%%%

Let $g(\bx)=g(x_1,x_2,\dots,x_k)$ be a quadratic polynomial with coefficients in the field of rational numbers $\q$.
Let $n$ be an integer, and let $R$ denote either the ring of integers $\z$, the ring of $p$-adic integers $\z_p$, or the field of real numbers $\mathbb{R}$.
We say that {\it $n$ is represented by $g$ over $R$} if the equation
$$
g(x_1,x_2,\dots,x_k)=n
$$
has a solution over $R$.
In this case, we write $n \to g$ over $R$.
When $R=\z$, we omit the phrase “over $\z$” and simply say that $n$ is represented by $g$.
We say that {\it $n$ is locally represented by $g$} if it is represented by $g$ over $\z_p$ for all primes $p$ and over $\mathbb{R}$.
Throughout, let $P$ denote the set of all prime numbers greater than $3$, and let $\Omega$ denote the set of all primes together with the infinite prime, so that $\Omega = P \cup \{2,3,\infty\}$.
Thus, $n$ is locally represented by $g$ if and only if it is represented by $g$ over $\z_p$ for all $p \in \Omega$, where we adopt the convention $\z_\infty = \mathbb{R}$.

Two quadratic polynomials $g_1(\bx)$ and $g_2(\bx)$ are said to be {\it equivalent} if
$$
g_1(\bx)=g_2(\bx T+\bx_0)
$$
for some $T \in GL_n(\z)$ and $\bx_0 \in \z^n$.
In this case, $n$ is (locally) represented by $g_1$ if and only if it is (locally, respectively) represented by $g_2$.

For an integer $m \ge 3$, define
$$
P_m(x)=\frac{(m-2)x^2-(m-4)x}{2}.
$$
An integer $n$ is called an {\it $m$-gonal number} (or a {\it polygonal number of order $m$}) if $n=P_m(u)$ for some $u \in \n$.
Fermat’s polygonal number theorem states that every positive integer is a sum of at most $m$ polygonal numbers of order $m$.
An integer $n$ is called a {\it generalized $m$-gonal number} if $n=P_m(v)$ for some $v \in \z$.
A quadratic polynomial of the form
$$
f(x_1,x_2,\dots,x_k)=a_1P_m(x_1)+\cdots+a_kP_m(x_k)\quad (a_i \in \n)
$$
is called a {\it $k$-ary sum of generalized $m$-gonal numbers} (or an {\it $m$-gonal form}).
An $m$-gonal form is said to be {\it universal} if it represents all nonnegative integers.
Gauss’s Eureka theorem states that the ternary triangular form $P_3(x_1)+P_3(x_2)+P_3(x_3)$ is universal.

We now consider the notion of regularity of an $m$-gonal form.
There are two possible definitions:
\begin{enumerate}[(i)]
\item $f$ is regular if it represents all integers that are locally represented.
\item $f$ is regular if it represents all nonnegative integers that are locally represented.
\end{enumerate}
In \cite{CO13}, it was shown that there are only finitely many primitive ternary triangular forms that are regular in sense (i).
Chan and Ricci \cite{CR} proved that there exist only finitely many primitive regular ternary complete quadratic polynomials of a fixed conductor.
As a corollary, they obtained the finiteness of regular ternary $m$-gonal forms (in sense (i)) for each fixed integer $m \ge 3$.

In this article, however, we adopt definition (ii).
Thus, an $m$-gonal form is said to be {\it regular} if it represents all nonnegative integers that are locally represented by it.
To justify this choice, note that
$$
\{P_m(u) : u \in \z\} \subset \n_0 \quad \text{whereas} \quad 
\{P_m(u) : u \in \mathbb{R}\} = \left\{ v \in \mathbb{R} : v \ge -\frac{(m-4)^2}{8(m-2)} \right\},
$$
where $\n_0 = \n \cup \{0\}$.
Thus, no $m$-gonal form can represent a negative integer, whereas it may locally represent negative integers.

To illustrate the difference, suppose we adopt definition (i).
Then the quaternary triangular form
$$
p_3(1,1,3,6)=P_3(x_1)+P_3(x_2)+3P_3(x_3)+6P_3(x_4)
$$
is universal (see \cite{BK}) but not regular.
Indeed, it locally represents every integer $\ge -\frac{1+1+3+6}{8}$, and hence in particular the integer $-1$.
More precisely, for $(x_1,x_2,x_3,x_4)\in \z_p^4$ defined by
$$
(x_1,x_2,x_3,x_4)=
\begin{cases}
\left(-\frac{1}{2},-\frac{1}{2},0,-\frac{1}{2}\right) & \text{if } p=3,\\[1em]
\left(0,0,-\frac{1}{3},-\frac{1}{3}\right) & \text{if } p\in \Omega \smallsetminus \{3\},
\end{cases}
$$
we have
$$
\frac{x_1(x_1+1)}{2}+\frac{x_2(x_2+1)}{2}
+3\frac{x_3(x_3+1)}{2}+6\frac{x_4(x_4+1)}{2}=-1.
$$
This example shows that any universal triangular form $p_3(a_1,\dots,a_k)$ with $a_1+\cdots+a_k>8$ cannot be regular in the sense of (i).
Moreover, the ternary triangular form $p_3(1,3,27)$ is excluded from the list of regular forms (see \cite[Table 5]{KO}), since it locally but not globally represents $-3$.

The main result of this article is as follows.

%%%%%%%%%%%%%%%%%%%%%%%%%%%%%%%%%%%%%%%%%%%%%%%%%%%%%%%%%%%%%%%%%%%%%%%%%%%%%%%%%%%%%%%%%%

\begin{thm} \label{thmmain}
Let $m$ be an integer greater than or equal to 3 for which there exists a regular ternary $m$-gonal form. Then we have
$$
m\le \begin{cases}35&\text{if}\ \ m\equiv 1\pmod 2,\ m\equiv 2\pmod 3,\\
147&\text{if}\ \ m\equiv 1\pmod 2,\ m\not\equiv 2\pmod 3,\\
38&\text{if}\ \ m\equiv 2\pmod 4,\ m\equiv 2\pmod 3,\\
142&\text{if}\ \ m\equiv 2\pmod 4,\ m\not\equiv 2\pmod 3,\\
188&\text{if}\ \ m\equiv 0\pmod 4,\ m\equiv 2\pmod 3,\\
712&\text{if}\ \ m\equiv 0\pmod 4,\ m\not\equiv 2\pmod 3.\end{cases}
$$
\end{thm}

%%%%%%%%%%%%%%%%%%%%%%%%%%%%%%%%%%%%%%%%%%%%%%%%%%%%%%%%%%%%%%%%%%%%%%%%%%%%%%%%%%%%%%%%%%

Any quadratic polynomial $g(\bx)$ with rational coefficients can be written in the form
$$
g(\bx)=Q(\bx)+\ell(\bx)+C,
$$
where $Q$ is a quadratic form, $\ell$ is a linear form, and $C$ is a constant.
Throughout this article, we assume that $Q$ is positive definite unless otherwise stated.
Let $B$ be the symmetric bilinear form associated with $Q$, that is,
$$
B(\bx,\by)=\frac{1}{2}\bigl(Q(\bx+\by)-Q(\bx)-Q(\by)\bigr).
$$
Then the matrix $(B(x_i,x_j))_{1\le i,j\le k} \in M_k(\q)$ is invertible.
Hence there exists a unique vector $\bv \in \q^k$ such that
$$
2B(\bx,\bv)=\ell(\bx).
$$
It follows that
$$
g(\bx)=Q(\bx+\bv)-Q(\bv)+C.
$$
We call $g$ {\it complete} if $C=Q(\bv)$, in which case
$$
g(\bx)=Q(\bx+\bv).
$$

Let $h=h(\bx)$ be a complete quadratic polynomial.
Then $h(\bx)\ge 0$ for all $\bx \in \q^k$.
We define
$$
\min(h)=\min \{ h(\bx) : \bx \in \z^k\}
$$
and call it the {\it minimum of $h$}.
We say that $h$ is {\it tight regular} if it represents all locally represented integers greater than or equal to $\min(h)$.

Throughout this article, we let $m \ge 3$ be an integer and define constants $\delta, c, d, \mu$ depending only on $m$ by
$$
\delta=\delta(m)=
\begin{cases}
4 & \text{if } m\equiv 1 \pmod{2},\\
2 & \text{if } m\equiv 2 \pmod{4},\\
1 & \text{if } m\equiv 0 \pmod{4},
\end{cases}
$$
and
\begin{align*}
c &= \delta \frac{m-2}{2},\\
d &= \delta \frac{m-4}{4},\\
\mu &= \delta c.
\end{align*}
Then $c$ and $d$ are relatively prime.

Fix $m \in \z_{\ge 3}$.
Let $a_1,a_2,a_3$ be positive integers, and let $n$ be a nonnegative integer.
With this notation, the equation
$$
a_1P_m(x_1)+a_2P_m(x_2)+a_3P_m(x_3)=n
$$
holds if and only if
$$
a_1(cx_1-d)^2+a_2(cx_2-d)^2+a_3(cx_3-d)^2
= \mu n + d^2(a_1+a_2+a_3).
$$
In Proposition \ref{propcorres}, we show that a ternary $m$-gonal form
$$
a_1P_m(x_1)+a_2P_m(x_2)+a_3P_m(x_3)
$$
is regular if and only if the complete quadratic polynomial
$$
a_1(cx_1-d)^2+a_2(cx_2-d)^2+a_3(cx_3-d)^2
$$
is tight regular.
The absence of a regular ternary complete quadratic polynomial of conductor $c$ does not preclude the existence of a regular (in sense (ii)) ternary $m$-gonal form.

The remainder of this article is organized as follows.
In Section~2, we review basic notions and useful lemmas concerning quadratic lattices and spaces, $\z$-cosets, Watson transformations, and related topics.
Section~3 is devoted to the proof of Theorem~\ref{thmmain}.
We briefly outline the strategy of the proof.
Let $m \ge 3$ be an integer for which there exists a regular ternary $m$-gonal form
$$
f=a_1'P_m(x_1)+a_2'P_m(x_2)+a_3'P_m(x_3).
$$
The existence of such a form implies the existence of a tight regular ternary $\z$-coset $L+\bv$ with conductor $c$.
Under Watson transformations, $L+\bv$ can be transformed into another tight regular ternary $\z$-coset $K+\bw$ with the same conductor $c$, but which locally represents more small positive integers than $L+\bv$.
Since $K+\bw$ is tight regular, every locally represented integer greater than or equal to its minimum is globally represented. 
In particular, this implies that $K+\bw$ represents many integers in any fixed interval.
On the other hand, as the conductor of $K+\bw$ increases, the number of integers (in a fixed interval) that it can represent decreases.
Combining these observations yields an upper bound for $c$, and hence an upper bound for $m$.

%%%%%%%%%%%%%%%%%%%%%%%%%%%%%%%%%%%%%%%%%%%%%%%%%%%%%%%%%%%%%%%%%%%%%%%%%%%%%%%%%%%%%%%%%%
%%%%%%%%%%%%%%%%%%%%%%%%%%%%%%%%%%%%%%%%%%%%%%%%%%%%%%%%%%%%%%%%%%%%%%%%%%%%%%%%%%%%%%%%%%
%%%%%%%%%%%%%%%%%%%%%%%%%%%%%%%%%%%%%%%%%%%%%%%%%%%%%%%%%%%%%%%%%%%%%%%%%%%%%%%%%%%%%%%%%%
%%%%%%%%%%%%%%%%%%%%%%%%%%%%%%%%%%%%%%%%%%%%%%%%%%%%%%%%%%%%%%%%%%%%%%%%%%%%%%%%%%%%%%%%%%
\section{Preliminaries}
%%%%%%%%%%%%%%%%%%%%%%%%%%%%%%%%%%%%%%%%%%%%%%%%%%%%%%%%%%%%%%%%%%%%%%%%%%%%%%%%%%%%%%%%%%
%%%%%%%%%%%%%%%%%%%%%%%%%%%%%%%%%%%%%%%%%%%%%%%%%%%%%%%%%%%%%%%%%%%%%%%%%%%%%%%%%%%%%%%%%%
%%%%%%%%%%%%%%%%%%%%%%%%%%%%%%%%%%%%%%%%%%%%%%%%%%%%%%%%%%%%%%%%%%%%%%%%%%%%%%%%%%%%%%%%%%
%%%%%%%%%%%%%%%%%%%%%%%%%%%%%%%%%%%%%%%%%%%%%%%%%%%%%%%%%%%%%%%%%%%%%%%%%%%%%%%%%%%%%%%%%%

We adopt the geometric language of quadratic spaces and lattices.
Let $F$ be either the field of rational numbers $\q$ or the field of $p$-adic numbers $\q_p$ for a prime $p$, and let $R=\z$ or $\z_p$ accordingly.
Let $V$ be a finite-dimensional quadratic space over $F$ with quadratic map $Q:V\to F$.
Throughout, we assume that the symmetric bilinear form $B:V\times V\to F$ defined by
$$
B(\bx,\by)=\frac{1}{2}\bigl(Q(\bx+\by)-Q(\bx)-Q(\by)\bigr)
$$
is nondegenerate.

An $R$-lattice is a finitely generated free $R$-submodule of $V$.
We always assume that an $R$-lattice $L$ is {\it integral}, that is, $B(\bx,\by)\in R$ for all $\bx,\by\in L$.
Let $L=R\bv_1+R\bv_2+\cdots+R\bv_k$ be an $R$-lattice on $V$.
The matrix $M_L=(B(\bv_i,\bv_j))_{1\le i,j\le k}$ is called the Gram matrix of $L$ with respect to the basis $\{\bv_1,\bv_2,\dots,\bv_k\}$.
The ideal of $R$ generated by $\{B(\bx,\by):\bx,\by\in L\}$ is called the {\it scale} of $L$ and is denoted by $\mathfrak{s}(L)$.
We call $L$ {\it primitive} if $\mathfrak{s}(L)=R$.
For $\alpha \in F$, let $L^{\alpha}$ denote the $R$-lattice $R\bv_1+R\bv_2+\cdots+R\bv_k$ in the scaled quadratic space $(V,Q^{\alpha})$, where $Q^{\alpha}(\bx)=\alpha Q(\bx)$.
For an $R$-lattice $K$, we write
$$
K=R\bw_1+R\bw_2+\cdots+R\bw_k\simeq A
$$
when $\{\bw_1,\bw_2,\dots,\bw_k\}$ is an $R$-basis of $K$ and $A=(B(\bw_i,\bw_j))_{1\le i,j\le k}$.

Let $L$ and $K$ be $R$-lattices.
If there exists a linear map $\sigma: FL\to FK$ such that
$$
B(\sigma(\bx),\sigma(\by))=B(\bx,\by)\quad \text{for all }\bx,\by\in FL
$$
and $\sigma(L)\subseteq K$, then we say that {\it $L$ is represented by $K$} and write $L\ra K$.
In this case, $\sigma$ is called a {\it representation of $L$ by $K$}.
If, in addition, $\sigma(L)=K$, then $L$ and $K$ are said to be {\it isometric}.

Throughout, we assume that a quadratic space $V$ over $\q$ is positive definite, that is, $Q(\bv)>0$ for all $\bv \in V\setminus\{\mathbf{0}\}$.
A $\z$-lattice $L$ is said to be {\it diagonal} if its Gram matrix with respect to some basis is diagonal.
A diagonal matrix with diagonal entries $a_1,a_2,\dots,a_k$ is denoted by $\langle a_1,a_2,\dots,a_k\rangle$.
By abuse of notation, we write $L=\langle a_1,a_2,\dots,a_k\rangle$ if the Gram matrix of $L$ with respect to some basis is equal to $\langle a_1,a_2,\dots,a_k\rangle$.

For a $\z$-lattice $L$ and a prime $p$, we denote by $L_p$ the $\z_p$-lattice $L\otimes \z_p$.
For a $\z_p$-lattice $K$, the set
$$
Q(K)=\{ Q(\bv)\in \z_p : \bv \in K\}
$$
can be determined using the results of \cite{OM1}.
For an odd prime $p$, let $\Delta_p$ denote a nonsquare unit in $\z_p$.
Any unexplained notation and terminology can be found in \cite{Ki} or \cite{OM}.

Let $V$ be a finite-dimensional quadratic space over $\q$.
A subset $U$ of $V$ is called a {\it $\z$-coset} on $V$ if
$$
U=L+\bv=\{\, \bx+\bv \in V : \bx \in L \,\}
$$
for some $\z$-lattice $L$ on $V$ and a vector $\bv \in V$.
Throughout, when we say that “$L+\bv$ is a $\z$-coset,” we always mean that $L$ is a $\z$-lattice and $\bv \in \q L = \q \otimes L$, so that $L+\bv$ is a $\z$-coset on the quadratic space $\q L$.
We call $L+\bv$ a $k$-ary $\z$-coset if $\operatorname{rank}(L)=k$.
A $\z$-coset $L+\bv$ is called {\it integral} if $Q(\bx+\bv)\in \z$ for all $\bx \in L$.
An integral $\z$-coset $L+\bv$ is called {\it primitive} if $\mathfrak{n}(L+\bv)=\z$, where $\mathfrak{n}(L+\bv)$ is the ideal of $\z$ generated by $\{Q(\bx+\bv) : \bx \in L\}$.
The ideal $\mathfrak{n}(L+\bv)$ is called the {\it norm} of $L+\bv$.
The smallest positive integer $\tilde{c}$ such that $\tilde{c}\bv \in L$ is called the {\it conductor of $L+\bv$}.
Thus, a $\z$-coset $L+\bv$ is a $\z$-lattice if and only if its conductor is equal to $1$.
For each prime $p$, $\z_p$-cosets are defined analogously.

Let $n$ be a nonnegative integer.
We say that $n$ is represented by $L+\bv$ if there exists a vector $\bx \in L$ such that $Q(\bx+\bv)=n$.
We say that $n$ is locally represented by $L+\bv$ if, for every prime number $p$, there exists a vector $\bx_p \in L_p$ such that $Q(\bx_p+\bv)=n$.

Let $L+\bv$ be a $\z$-coset of conductor greater than $1$.
We define
$$
\min(L+\bv)=\min\{Q(\bx+\bv) : \bx\in L\}.
$$
We say that $L+\bv$ is {\it tight regular} if it represents all locally represented integers greater than or equal to $\min(L+\bv)$.

Let $n$ be a nonnegative integer, and let 
$$
f=a_1P_m(x_1)+a_2P_m(x_2)+a_3P_m(x_3)
$$
be a ternary $m$-gonal form.
As shown in the introduction, $n$ is represented by $f$ if and only if
$$
\mu n+d^2(a_1+a_2+a_3)
$$
is represented by the ternary complete quadratic polynomial
$$
g=a_1(cx_1-d)^2+a_2(cx_2-d)^2+a_3(cx_3-d)^2.
$$
Define a ternary $\z$-lattice $L$ and a vector $\bv \in \q L$ by
$$
L=\z \be_1+\z \be_2+\z \be_3 \simeq \langle a_1c^2,a_2c^2,a_3c^2\rangle,
\quad
\bv=-\frac{d}{c}(\be_1+\be_2+\be_3).
$$
Then
$$
a_1(cx_1-d)^2+a_2(cx_2-d)^2+a_3(cx_3-d)^2=\mu n+d^2(a_1+a_2+a_3)
$$
if and only if
$$
Q(x_1\be_1+x_2\be_2+x_3\be_3+\bv)=\mu n+d^2(a_1+a_2+a_3).
$$
Thus, $n$ is represented by $f$ if and only if $\mu n + d^2(a_1+a_2+a_3)$ is represented by the $\z$-coset $L+\bv$.
Consequently, $g$ is tight regular if and only if $L+\bv$ is tight regular.

Let $L$ be a $\z$-lattice and let $N$ be a positive integer.
A {\it Watson transformation of $L$ modulo $N$} is the sublattice of $L$ defined by
$$
\Lambda_N(L)=\{ \bx \in L : Q(\bx+\bz)\equiv Q(\bz)\ \text{for all }\bz \in L\}.
$$
We denote by $\lambda_N(L)$ the $\z$-lattice obtained from $\Lambda_N(L)$ by scaling $\q L$ by a suitable rational number so that $\mathfrak{s}(\lambda_N(L))=\mathfrak{s}(L)$.

Let $p$ be a prime, and let
$$
L=\z \be_1+\z \be_2+\z \be_3 \simeq \langle a_1,p^{s_2}a_2,p^{s_3}a_3\rangle
$$
be a ternary diagonal $\z$-lattice, where $(p,a_1a_2a_3)=1$ and $0\le s_2\le s_3$.
Then
$$
\Lambda_p(L)=
\begin{cases}
\z (p\be_1)+\z \be_2+\z \be_3 \simeq \langle p^2a_1,p^{s_2}a_2,p^{s_3}a_3\rangle & \text{if } s_2>0,\\
\z (p\be_1)+\z (p\be_2)+\z \be_3 \simeq \langle p^2a_1,p^2a_2,p^{s_3}a_3\rangle & \text{if } p>2,\ s_2=0<s_3.
\end{cases}
$$
If $p=2$, $s_2=0$, $s_3\ge 2$, and $a_1a_2\equiv 1 \pmod{4}$, then
$$
\Lambda_4(L)=\z (2\be_1)+\z (2\be_2)+\z \be_3 \simeq \langle 4a_1,4a_2,2^{s_3}a_3\rangle,
$$
and hence
$$
\lambda_4(L)=\z (2\be_1)+\z (2\be_2)+\z \be_3 \simeq \langle a_1,a_2,2^{s_3-2}a_3\rangle.
$$
We refer the reader to \cite{CE,CO03} for further details on Watson transformations.

Let $K$ be a diagonal ternary $\z$-lattice.
For an odd prime $p$, we say that $K$ is {\it $p$-stable} if
$$
\langle 1,-1\rangle \ra K_p \quad \text{or} \quad K_p \simeq \langle 1,-\Delta_p\rangle \perp \langle p\epsilon_p\rangle,
$$
where $\epsilon_p\in \z_p^{\times}$.
In the first case, $K_p$ is isotropic and $Q(K_p)=\z_p$.
In the latter case, $K_p$ is anisotropic and
$$
\z_p^{\times} \subset Q(K_p)=\z_p \smallsetminus \{ p\epsilon_p\Delta_p \gamma^2 : \gamma \in \z_p\}.
$$
We say that $K$ is {\it $2$-stable} if
$$
K_2 \text{ is unimodular} \quad \text{or} \quad \langle 1,3\rangle \ra K_2 \quad \text{or} \quad \langle 1,7\rangle \ra K_2.
$$
If $K$ is $2$-stable, then
$$
\{ \gamma \in \z_2 : \ord_2(\gamma)\equiv 0 \pmod{2}\} \subseteq Q(K_2),
$$
except when $K_2 \simeq \begin{pmatrix}2&1\\1&2\end{pmatrix}\perp \langle \epsilon \rangle$, in which case
$$
Q(K_2)=\z_2 \setminus \{ (\epsilon+4)\gamma^2 : \gamma \in \z_2\}.
$$
For background on local representations of lattices, we refer the reader to \cite{OM1}.
For a prime $q$, we say that $K$ is {\it $q$-unstable} if it is not $q$-stable.

The following lemma is a slight modification of \cite[Proposition 4.6]{CR}, and its proof is essentially identical.
In Lemma \ref{lemlambda}, the $\z$-coset $\lambda_q(L)+p^j\bv$ coincides with $\Lambda_q(L)+p^j\bv$ as a set, but is equipped with a scaled quadratic map so that
$$
p^j\bv \in \q(\Lambda_q(L))=\q(\lambda_q(L)).
$$

%%%%%%%%%%%%%%%%%%%%%%%%%%%%%%%%%%%%%%%%%%%%%%%%%%%%%%%%%%%%%%%%%%%%%%%%%%%%%%%%%%%%%%%%%%

\begin{lem} \label{lemlambda}
Let $L+\bv$ be a primitive ternary $\z$-coset with conductor $\tilde{c}$, and let $p$ be a prime not dividing $\tilde{c}$.
Let $j$ be the smallest positive integer such that $p^j\equiv 1 \pmod{\tilde{c}}$.
Assume that $L$ is diagonal and $p$-unstable.
If $L+\bv$ is tight regular, then the $\z$-coset $\lambda_q(L)+p^j\bv$ is primitive and tight regular,
where $q=4$ if $p=2$ and the rank of the unimodular component of $L_2$ is equal to $2$, and $q=p$ otherwise.
Moreover, the conductor of $\lambda_q(L)+p^j\bv$ is equal to $\tilde{c}$, and $\lambda_q(L)$ is diagonal.
\end{lem}

\begin{proof}
First, it is clear that $\lambda_q(L)$ is diagonal.
Let $s\in \{1,2\}$ be such that $\lambda_q(L)=\Lambda_q(L)^{p^{-s}}$.
Denote by $Q$ the quadratic map on $\q L$, so that $Q^{p^{-s}}$ is the quadratic map on $\q (\lambda_q(L))$.
Since $\Lambda_q(L)\subseteq L$ and $(p^j-1)\bv\in L$, we have
$$
\Lambda_q(L)+p^j\bv=\Lambda_q(L)+(p^j-1)\bv+\bv\subseteq L+\bv.
$$
Thus we have
$$
\min(L+\bv)\le \min(\Lambda_q(L)+p^j\bv)=p^s\min(\lambda_q(L)+p^j\bv).
$$

To prove that $\lambda_q(L)+p^j\bv$ is tight regular, let $n$ be an integer greater than or equal to $\min(\lambda_q(L)+p^j\bv)$ such that $n$ is locally represented by $\lambda_q(L)+p^j\bv$.
Then $p^s n$ is locally represented by $\Lambda_q(L)+p^j\bv$, and hence by $L+\bv$.
Moreover, $p^s n \ge \min(L+\bv)$.
Therefore, by assumption, $p^s n$ is (globally) represented by $L+\bv$.
This implies that there exists a vector $\bx \in L$ such that $Q(\bx+\bv)=p^sn$.

We assert that $\bx+\bv \in \Lambda_q(L)+p^j\bv$.
For a prime $r$, note that
$$
\Lambda_q(L)_r+p^j\bv=\begin{cases}L_r+\bv&\text{if}\ r\mid \tilde{c},\\
\Lambda_q(L)_p&\text{if}\ r=p,\\
L_r&\text{if}\ r\nmid p\tilde{c}.\end{cases}
$$
Hence, it is clear that $\bx+\bv \in \Lambda_q(L)_r+p^j\bv$ whenever $r\neq p$.
When $r=p$, the assumption that $L$ is diagonal $p$-unstable implies that
$$
\Lambda_q(L_p)=\{ \by \in L_p : Q(\by)\equiv 0\pmod q\}.
$$
Since $p^j\bv\in L_p$ and $Q(p^j\bv)$ is divisible by $q$, we conclude that $p^j\bv \in \Lambda_q(L_p)=\Lambda_q(L)_p$.
This shows that $\bx+\bv \in \Lambda_q(L)_r+p^j\bv$ for all primes $r$, and thus the assertion follows.
By this assertion, $p^s n$ is represented by $\Lambda_q(L)+p^j\bv$.
Therefore, $n$ is represented by $\lambda_q(L)+p^j\bv$.

Now, we show that $\mathfrak{n}(\lambda_q(L)+p^j\bv)=\z$.
Note that $\mathfrak{n}(L_r+\bv)=\z_r$ for all primes $r$, and in particular, $\mathfrak{n}(L_r)=\z_r$ whenever $r\nmid \tilde{c}$. 
Thus, $\mathfrak{n}(\lambda_q(L)_r+p^j\bv)=\z_r$ for every prime $r\neq p$.
Observe that
$$
\mathfrak{n}(\Lambda_q(L)_p+p^j\bv)=\mathfrak{n}(\Lambda_q(L)_p)=p^s\mathfrak{n}(\lambda_q(L)_p)=p^s\mathfrak{n}(L_p)=p^s\z_p.
$$
This shows that $\mathfrak{n}(\lambda_q(L)+p^j\bv)=\z$, which implies that $\lambda_q(L)+p^j\bv$ is primitive.
The conductor of $\lambda_q(L)+p^j\bv$ is clearly equal to $\tilde{c}$.
This completes the proof.
\end{proof}

%%%%%%%%%%%%%%%%%%%%%%%%%%%%%%%%%%%%%%%%%%%%%%%%%%%%%%%%%%%%%%%%%%%%%%%%%%%%%%%%%%%%%%%%%%

\begin{rmk}
Let $p$ be a prime with $(p,2c)=1$.
Let $a_1,a_2,a_3$ be positive integers such that $(p,a_1)=1$ and
$a_2\equiv a_3\equiv 0\!\pmod{p}$.
Define
$$
s=\begin{cases}
2 & \text{if } a_2\equiv a_3\equiv 0\!\pmod{p^2},\\
1 & \text{otherwise}.
\end{cases}
$$
Then
$$
\lambda_p(\langle a_1c^2,a_2c^2,a_3c^2\rangle)\simeq
\langle p^{2-s}a_1c^2,\,p^{-s}a_2c^2,\,p^{-s}a_3c^2\rangle.
$$
Let
$$
L=\z \be_1+\z \be_2+\z \be_3 \simeq \langle a_1c^2, a_2c^2, a_3c^2\rangle,
\quad
\bv=-\frac{d}{c}(\be_1+\be_2+\be_3).
$$
Let $j$ be the smallest positive integer such that $p^j\equiv 1\!\pmod{c}$.
Consider the $\z$-coset $\Lambda_p(L)+p^j\bv$.
Note that
$$
\Lambda_p(L)=\z(p\be_1)+\z\be_2+\z\be_3.
$$
For $(x_1,x_2,x_3)\in \z^3$, we have
\begin{align*}
&Q(x_1p\be_1+x_2\be_2+x_3\be_3+p^j\bv)\\
&=Q\Bigl(\bigl(px_1-p^j\tfrac{d}{c}\bigr)\be_1
+\bigl(x_2-p^j\tfrac{d}{c}\bigr)\be_2
+\bigl(x_3-p^j\tfrac{d}{c}\bigr)\be_3\Bigr)\\
&=a_1(pcx_1-p^jd)^2+a_2(cx_2-p^jd)^2+a_3(cx_3-p^jd)^2\\
&=p^2a_1(cx_1-p^{j-1}d)^2
+a_2(cx_2-p^jd)^2
+a_3(cx_3-p^jd)^2.
\end{align*}
Since the quadratic map on $\q(\lambda_p(L))$ is $Q^{p^{-s}}$, we obtain
\begin{align*}
&Q^{p^{-s}}(x_1p\be_1+x_2\be_2+x_3\be_3+p^j\bv)\\
&=p^{2-s}a_1(cx_1-p^{j-1}d)^2
+p^{-s}a_2(cx_2-p^jd)^2
+p^{-s}a_3(cx_3-p^jd)^2.
\end{align*}
If $p\not\equiv \pm1\!\pmod{c}$, then $p^{j-1}d\not\equiv \pm d\!\pmod{c}$.
Hence the representations of integers by the $\z$-coset
$\lambda_p(L)+p^j\bv$ do not correspond to those by the $m$-gonal form
$$
p^{2-s}a_1P_m(x_1)+p^{-s}a_2P_m(x_2)+p^{-s}a_3P_m(x_3).
$$
This phenomenon is not limited to the present case.
\end{rmk}

%%%%%%%%%%%%%%%%%%%%%%%%%%%%%%%%%%%%%%%%%%%%%%%%%%%%%%%%%%%%%%%%%%%%%%%%%%%%%%%%%%%%%%%%%%

\begin{lem} \label{lemz23}
Let $K$ be a $2$-stable diagonal ternary $\z$-lattice.
Let $u$ be an odd integer and $l$ be an arbitrary integer.
Then there exists an integer $\nu \in \{0,1,2\}$ such that 
$u(4n+\nu)+l$ is represented by $K_2$ for every integer $n$.
If, in addition, $K$ is $3$-stable and $(u,6)=1$, then there exists an integer 
$\nu \in \{0,1,2,3,4\}$ such that 
$u(12n+\nu)+l$ is represented by both $K_2$ and $K_3$ for every integer $n$.
\end{lem}
 
\begin{proof}
By assumption, since $K$ is 2-stable, at least two congruence classes modulo 4 are represented by $K_2$.
That is, there exist integers $k_1,k_2$ with $0\le k_1<k_2\le 3$ such that every integer $k\equiv k_1$ or $k_2 \pmod{4}$ is represented by $K_2$.
The first assertion follows immediately.

Now assume further that $K$ is 3-stable and that $(u,6)=1$.
Then $\z_3^{\times}\subset Q(K_3)$, so every integer that is not divisible by 3 and is congruent to $k_1$ or $k_2$ modulo 4 is represented by both $K_2$ and $K_3$.
Suppose that
\[
l\equiv k_1\pmod{4}\quad \text{or}\quad l\equiv k_2\pmod{4}.
\]
Then both of $u(12n+0)+l$ and $u(12n+4)+l$ are represented by $K_2$.
Since either $u(12n+0)+l$ or $u(12n+4)+l$ is not divisible by 3, at least one of them is represented by $K_3$.
Now, suppose that $l\not\equiv k_i\pmod 4$ for $i=1,2$.
Then there exist two distinct integers $\nu_1$ and $\nu_2$ with $1\le \nu_1,\nu_2\le 3$ such that
$$
u(12n+\nu_i)+l\equiv k_i\pmod 4
$$
for $i=1,2$.
Hence, $u(12n+\nu_i)+l$ is represented by $K_2$ for $i=1,2$.
Since
\[
u(12n+\nu_1)+l\not\equiv u(12n+\nu_2)+l\pmod 3,
\]
at least one of them is represented by $K_3$.
This completes the proof.
\end{proof}

%%%%%%%%%%%%%%%%%%%%%%%%%%%%%%%%%%%%%%%%%%%%%%%%%%%%%%%%%%%%%%%%%%%%%%%%%%%%%%%%%%%%%%%%%%

\begin{lem} \label{lem1}
Let $p$ be a prime greater than $3$, and let $u$ be a positive integer coprime to $p$.
Let $a_1,a_2,a_3,\alpha_1,\alpha_2,$ and $\alpha_3$ be positive integers.
Assume that the ternary diagonal $\z$-lattice $\langle a_1,a_2,a_3\rangle$ is isometric to $\langle 1,-\Delta_p\rangle \perp \langle p\epsilon_p\rangle$ over $\z_p$, where $\epsilon_p\in \z_p^{\times}$.
Then there exists an integer $v$ such that
\begin{enumerate}[(i)]
\item $0<v<p^2$,
\item $uv+a_1\alpha_1^2+a_2\alpha_2^2 \nra \langle a_1,a_2\rangle$ over $\z_p$,
\item $uv+\tail \ra \langle a_1,a_2,a_3\rangle$ over $\z_p$,
\item $\max\bigl(\text{\rm ord}_p(uv+a_1\alpha_1^2+a_2\alpha_2^2),\,\text{\rm ord}_p(uv+\tail)\bigr)\le 1$.
\end{enumerate}
\end{lem}

\begin{proof}
First, suppose that $p$ divides both $a_3$ and $\alpha_3$.
Choose $v$ with $0\le v<p^2$ such that
\[
uv+a_1\alpha_1^2+a_2\alpha_2^2 \equiv a_3 \pmod{p^2}.
\]
Then $uv+a_1\alpha_1^2+a_2\alpha_2^2 \in p\z_p^{\times}$, and hence it is not represented by $\langle a_1,a_2\rangle$ over $\z_p$, since $\langle a_1,a_2\rangle \simeq \langle 1,-\Delta_p\rangle$ over $\z_p$.
On the other hand, by the Local Square Theorem,
\[
uv+\tail \in a_3(\z_p^{\times})^2,
\]
and hence it is represented by $\langle a_1,a_2,a_3\rangle$ over $\z_p$.

Next, suppose that $p$ divides $a_3$ but does not divide $\alpha_3$.
Since $p\ge 5$, there exists an integer $t$ with $1\le t\le p-1$ such that
\[
t \not\equiv \alpha_3^2 \pmod{p}
\quad \text{and} \quad
\left(\frac{t}{p}\right)=1.
\]
Choose $v$ with $0\le v<p^2$ such that
\[
uv+a_1\alpha_1^2+a_2\alpha_2^2 \equiv a_3(t-\alpha_3^2) \pmod{p^2}.
\]
Then $uv+a_1\alpha_1^2+a_2\alpha_2^2 \in p\z_p^{\times}$, and hence it is not represented by $\langle a_1,a_2\rangle$ over $\z_p$.
On the other hand,
\[
uv+\tail \equiv a_3 t \pmod{p^2},
\]
and hence
\[
uv+\tail \in a_3 t (\z_p^{\times})^2 = a_3(\z_p^{\times})^2.
\]
Therefore $uv+\tail$ is represented by $\langle a_1,a_2,a_3\rangle$ over $\z_p$.

Finally, suppose that $p$ divides $a_1a_2$.
Without loss of generality, assume that $p$ divides $a_2$.
Since $p\ge 5$, there exists an integer $a'$ with $1\le a'\le p-1$ such that
\[
a' \not\equiv -a_3\alpha_3^2 \pmod{p}
\quad \text{and} \quad
\left(\frac{a'}{p}\right) = -\left(\frac{a_1}{p}\right).
\]
Choose $v$ with $0\le v\le p-1$ such that
\[
uv+a_1\alpha_1^2+a_2\alpha_2^2 \equiv a' \pmod{p}.
\]
Then $uv+a_1\alpha_1^2+a_2\alpha_2^2 \in a_1\Delta_p \z_p^{\times}$, and hence it is not represented by $\langle a_1,a_2\rangle$ over $\z_p$.
On the other hand,
\[
uv+\tail \in \z_p^{\times},
\]
and hence it is represented by $\langle a_1,a_2,a_3\rangle$ over $\z_p$.

Note that $v \ne 0$ in all of the above cases, since $u\cdot 0 + a_1\alpha_1^2 + a_2\alpha_2^2$ is represented by $\langle a_1,a_2\rangle$ over $\z_p$.
This completes the proof.
\end{proof}

%%%%%%%%%%%%%%%%%%%%%%%%%%%%%%%%%%%%%%%%%%%%%%%%%%%%%%%%%%%%%%%%%%%%%%%%%%%%%%%%%%%%%%%%%%

\begin{lem} \label{lem2}
Let $s>1$ be an integer, and let $5\le p_1<p_2<\cdots <p_s$ be prime numbers. 
Let $u$ be an integer with $(u,p_1p_2\cdots p_s)=1$, and let $v$ be an arbitrary integer.
Then there exists an integer $n$ with $0\le n<(s+4)2^{s-2}$ such that 
$(un+v,p_1p_2\cdots p_s)=1$.
\end{lem}

\begin{proof}
This follows from \cite[Lemma~3]{KKO} together with the assumption that $p_1\ge 5$.
\end{proof}

%%%%%%%%%%%%%%%%%%%%%%%%%%%%%%%%%%%%%%%%%%%%%%%%%%%%%%%%%%%%%%%%%%%%%%%%%%%%%%%%%%%%%%%%%%

For $i=1,2,\dots$, let $r_i$ denote the $i$-th smallest element of $P$, so that $r_1=5$, $r_2=7$, $r_3=11$, etc.
In the following proposition, the 13 inequalities are listed in the order in which they are used in the proof of Theorem~\ref{thmmain}.

%%%%%%%%%%%%%%%%%%%%%%%%%%%%%%%%%%%%%%%%%%%%%%%%%%%%%%%%%%%%%%%%%%%%%%%%%%%%%%%%%%%%%%%%%%

\begin{prop} \label{propproduct}
Let $t$ be an integer.
\begin{enumerate} [(i)]
\item If $t\ge 16$, then we have $r_7r_8\cdots r_t>((t+3)2^{t-3})^3$.
\item If $t\ge 7$, then we have $r_3r_4\cdots r_t>45\cdot 49\cdot (t+3)2^{t-3}$.
\item If $t\ge 5$, then we have $r_3r_4\dots r_t>2\cdot 18\cdot (t+3)2^{t-3}$.
\item If $t\ge 18$, then we have $r_7r_8\cdots r_t>(3(t+3)2^{t-3}-1)^3$.
\item If $t\ge 9$, then we have $r_3r_4\cdots r_t>142\cdot 304\cdot (3(t+3)2^{t-3}-1)$.
\item If $t\ge 8$, then we have $r_3r_4\cdots r_t>49\cdot 142\cdot (3(t+3)2^{t-3}-1)$.
\item If $t\ge 7$, then we have $r_3r_4\cdots r_t>9\cdot 88\cdot (3(t+3)2^{t-3}-1)$.
\item If $t\ge 18$, then we have $r_7r_8\cdots r_t>(4(t+3)2^{t-3}-1)^3$.
\item If $t\ge 9$, then we have $r_3r_4\cdots r_t>190\cdot 294\cdot (4(t+3)2^{t-3}-1)$.
\item If $t\ge 8$, then we have $r_3r_4\cdots r_t>66\cdot 110\cdot (4(t+3)2^{t-3}-1)$.
\item If $t\ge 7$, then we have $r_3r_4\cdots r_t>3\cdot 90\cdot (4(t+3)2^{t-3}-1)$.
\item If $t\ge 20$, then we have $r_7r_8\cdots r_t>(12(t+3)2^{t-3}-7)^3$.
\item If $t\ge 10$, then we have $r_3r_4\cdots r_t>580\cdot 1492\cdot (12(t+3)2^{t-3}-7)$.
\end{enumerate}
\end{prop}

\begin{proof}
Since the proofs are quite similar, we provide only the proof of (i).
For $t=16$, the inequality can be verified by direct computation:
\begin{align*}
r_7r_8\cdots r_{16}=12091972151626183(\simeq 1.2\cdot 10^{16}),\\
((16+3)2^{16-3})^3=3770775127457792(\simeq 3.8\cdot 10^{15}).
\end{align*}
For any integer $u \ge 16$, we have
$$
\dfrac{((u+4)2^{u-2})^3}{((u+3)2^{u-3})^3}\le \left( \dfrac{40}{19}\right)^3<10<r_{u+1}.
$$
Hence, by induction,
$$
r_7r_8\cdots r_t>((t+3)2^{t-3})^3
$$
for all integers $t \ge 16$.
Note that the five implications (1) (vii) $\Rightarrow$ (ii), (2) (vii) $\Rightarrow$ (xi), (3) (x) $\Rightarrow$ (vi), (4) (ix) $\Rightarrow$ (v) and (5) (viii) $\Rightarrow$ (iv) are immediate.
\end{proof}

%%%%%%%%%%%%%%%%%%%%%%%%%%%%%%%%%%%%%%%%%%%%%%%%%%%%%%%%%%%%%%%%%%%%%%%%%%%%%%%%%%%%%%%%%%

\begin{lem} \label{lemzp}
Let $p$ be an odd prime, let $u$ be a positive integer coprime to $p$, and let $v$ be an arbitrary integer.
Let $L$ be a diagonal $p$-stable ternary $\z$-lattice.
Then, for any positive integer $s$, we have
$$
\left| \{ 1\le n\le p^s : un+v \nra L \ \text{over } \z_p \} \right|
\le
\begin{cases}
\dfrac{p^s+p+2}{2p+2} & \text{if } s\equiv 1 \pmod{2}, \\[1em]
\dfrac{p^s+2p+1}{2p+2} & \text{if } s\equiv 0 \pmod{2}.
\end{cases}
$$
\end{lem}

\begin{proof}
See \cite[Lemma 2.5]{K1}.
\end{proof}

%%%%%%%%%%%%%%%%%%%%%%%%%%%%%%%%%%%%%%%%%%%%%%%%%%%%%%%%%%%%%%%%%%%%%%%%%%%%%%%%%%%%%%%%%%

Let $p$ be an odd prime.
For a positive integer $s$, define
$$
\psi_p(p^s)=\begin{cases}
\dfrac{p^s+p+2}{2p+2}&\text{if}\ \ s\ \text{is odd},\\[0.7em]
\dfrac{p^s+2p+1}{2p+2}&\text{if}\ \ s\ \text{is even}.\end{cases}
$$
Let $n\in \n$, and let $b_eb_{e-1}\cdots b_0{}_{(p)}$ be the base-$p$ representation of $n$, that is,
$$
n=b_ep^e+b_{e-1}p^{e-1}+\cdots+b_1p+b_0,
$$
where the $b_i$ are integers satisfying $0\le b_i\le p-1$ and $b_e\neq 0$.
We define
$$
\psi_p(n)=\begin{cases}\sum_{1\le s\le e} b_s\psi_p(p^s)&\text{if}\ \ b_0=0,\\
1+\sum_{1\le s\le e} b_s\psi_p(p^s)&\text{otherwise}.\end{cases}
$$
Then, by Lemma \ref{lemzp}, we have
\begin{equation} \label{eqpsi}
\vert \{ 1\le n\le n' : un+v\nra L\ \text{over}\ \z_p\} \le \psi_p(n')
\end{equation}
for any diagonal $p$-stable ternary $\z$-lattice $L$ and any positive integer $n'$.

%%%%%%%%%%%%%%%%%%%%%%%%%%%%%%%%%%%%%%%%%%%%%%%%%%%%%%%%%%%%%%%%%%%%%%%%%%%%%%%%%%%%%%%%%%

Recall that $P$ denotes the set of all primes greater than or equal to 5.

%%%%%%%%%%%%%%%%%%%%%%%%%%%%%%%%%%%%%%%%%%%%%%%%%%%%%%%%%%%%%%%%%%%%%%%%%%%%%%%%%%%%%%%%%%

\begin{lem} \label{lempsi}
Let $n$ be a positive integer.
Under the notation given above, we have
\begin{enumerate} [(i)]
\item $\vert\{ p\in P : \psi_p(n)>1\}\vert<\infty$,
\vskip 5pt
\item $\psi_p(n)\le \left\lceil \dfrac{n}{p}\right\rceil$ for any $p\in P$, where $\lceil \cdot \rceil$ denotes the ceiling function.
\end{enumerate}
\end{lem}

\begin{proof}
See \cite[Lemma 2.6]{K1}.
\end{proof}

For positive integers $n$ and $s$, we define
$$
\eta(n,s)=\min \left\{ n-\sum_{p\in P'}\psi_p(n) : P'\subset P,\ \vert P'\vert =s \right\}.
$$
By Lemma \ref{lempsi}(i), the value $\eta(n,s)$ is well defined for all positive integers $n$ and $s$.
We illustrate how to compute $\eta(48,8)$ as an example.
First, we compute $\psi_p(48)$ for several primes $p$:
$$
\begin{array}{c}
\psi_5(48)=8,\ \ \psi_7(48)=7,\ \ \psi_{11}(48)=5,\ \ \psi_{13}(48)=4,\\[0.3em]
\psi_{17}(48)=3,\ \ \psi_{19}(48)=3,\ \ \psi_{23}(48)=3,\ \ \psi_{29}(48)=2.
\end{array}
$$
Meanwhile, by Lemma \ref{lempsi}(ii), we have $\psi_p(48)\le 2$ for every prime $p>29$.
Hence,
$$
\eta(48,8)=48-(8+7+5+4+3+3+3+2)=13.
$$
We list the values of $\eta(n,s)$ for several pairs $(n,s)$ in Table \ref{tableeta}.

%%%%%%%%%%%%%%%%%%%%%%%%%%%%%%%%%%%%%%%%%%%%%%%%%%%%%%%%%%%%%%%%%%%%%%%%%%%%%%%%%%%%%%%%%%

\begin{table} [ht]
\begin{tabular}{c|c|c|c}
\hline
$\eta(17,8)=2$ & $\eta(19,6)=5$ & $\eta(20,4)=9$ & $\eta(25,6)=9$\\[0.2em] 
$\eta(25,7)=7$ & $\eta(28,8)=7$ & $\eta(30,7)=9$ & $\eta(48,8)=13$\\[0.2em]
$\eta(49,15)=5$ & $\eta(49,17)=3$ & $\eta(50,15)=7$ & $\eta(50,19)=3$\\[0.2em]
$\eta(60,9)=19$ & $\eta(74,17)=7$ & $\eta(102,17)=17$ & $\eta(125,19)=25$\\
\hline
\end{tabular}
\caption{$\eta(n,s)$ for some pair $(n,s)$}
\label{tableeta}
\end{table}

%%%%%%%%%%%%%%%%%%%%%%%%%%%%%%%%%%%%%%%%%%%%%%%%%%%%%%%%%%%%%%%%%%%%%%%%%%%%%%%%%%%%%%%%%%
%%%%%%%%%%%%%%%%%%%%%%%%%%%%%%%%%%%%%%%%%%%%%%%%%%%%%%%%%%%%%%%%%%%%%%%%%%%%%%%%%%%%%%%%%%
%%%%%%%%%%%%%%%%%%%%%%%%%%%%%%%%%%%%%%%%%%%%%%%%%%%%%%%%%%%%%%%%%%%%%%%%%%%%%%%%%%%%%%%%%%
%%%%%%%%%%%%%%%%%%%%%%%%%%%%%%%%%%%%%%%%%%%%%%%%%%%%%%%%%%%%%%%%%%%%%%%%%%%%%%%%%%%%%%%%%%
\section{A Proof of the Main Theorem}
%%%%%%%%%%%%%%%%%%%%%%%%%%%%%%%%%%%%%%%%%%%%%%%%%%%%%%%%%%%%%%%%%%%%%%%%%%%%%%%%%%%%%%%%%%
%%%%%%%%%%%%%%%%%%%%%%%%%%%%%%%%%%%%%%%%%%%%%%%%%%%%%%%%%%%%%%%%%%%%%%%%%%%%%%%%%%%%%%%%%%
%%%%%%%%%%%%%%%%%%%%%%%%%%%%%%%%%%%%%%%%%%%%%%%%%%%%%%%%%%%%%%%%%%%%%%%%%%%%%%%%%%%%%%%%%%
%%%%%%%%%%%%%%%%%%%%%%%%%%%%%%%%%%%%%%%%%%%%%%%%%%%%%%%%%%%%%%%%%%%%%%%%%%%%%%%%%%%%%%%%%%

Recall the definitions of $\delta,c,d$, and $\mu$ for $m\in \z_{\ge 3}$.

%%%%%%%%%%%%%%%%%%%%%%%%%%%%%%%%%%%%%%%%%%%%%%%%%%%%%%%%%%%%%%%%%%%%%%%%%%%%%%%%%%%%%%%%%%

\begin{prop} \label{propcorres}
Let $m\ge 3$ be an integer, and let
$$
f=f(x_1,x_2,x_3)=a_1P_m(x_1)+a_2P_m(x_2)+a_3P_m(x_3)
$$
be a ternary $m$-gonal form.
Let $L=\z \be_1+\z \be_2+\z \be_3$ be a $\z$-lattice whose Gram matrix is given by $(B(\be_i,\be_j))_{1\le i,j\le 3}=\langle a_1c^2,a_2c^2,a_3c^2\rangle$ and let $\bv=-\dfrac{d}{c}(\be_1+\be_2+\be_3)\in \q L$.
Then $f$ is regular if and only if the $\z$-coset $L+\bv$ is tight regular.
\end{prop}

\begin{proof}
First, we observe that $\min(L+\bv)=d^2(a_1+a_2+a_3)$ since $0<d<\dfrac{c}{2}$.
We begin with the ``if" direction.
Let $n$ be a nonnegative integer that is locally represented by $f$,
and define
$$
\beta(n)=\mu n+d^2(a_1+a_2+a_3).
$$
Then $\beta(n)$ is locally represented by $L+\bv$, and $\beta(n)\ge \min(L+\bv)$.
Since $L+\bv$ is tight regular, it follows that $\beta(n)$ is represented by $L+\bv$.
Consequently, $n$ is represented by $f$.
This completes the proof of the ``if" direction.

Next, we prove the ``only if" direction.
Let $N\ge \min(L+\bv)$ be an integer locally represented by $L+\bv$.
Then
$$
a_1(cx_1-d)^2+a_2(cx_2-d)^2+a_3(cx_3-d)^2=N
$$
is solvable over $\z_p$ for every prime $p$, and in particular, for every prime divisor $p$ of $\mu$.
From this, it follows that
$$
N\equiv d^2(a_1+a_2+a_3)\pmod{\mu}.
$$
Since $N\ge \min(L+\bv)$, we may write
$$
N=\mu n'+d^2(a_1+a_2+a_3)
$$
for some nonnegative integer $n'$.
Because $N$ is locally represented by $L+\bv$, it follows that $n'$ is locally represented by $f$.
By the regularity of $f$, $n'$ is represented by $f$.
Consequently, $N$ is represented by $L+\bv$.
This completes the proof of the ``only if" direction, and hence the proposition.
\end{proof}

%%%%%%%%%%%%%%%%%%%%%%%%%%%%%%%%%%%%%%%%%%%%%%%%%%%%%%%%%%%%%%%%%%%%%%%%%%%%%%%%%%%%%%%%%%

Although the following result arises naturally in the proof of Theorem~\ref{thmmain}, we state it separately as a lemma for ease of reference in a subsequent paper.

%%%%%%%%%%%%%%%%%%%%%%%%%%%%%%%%%%%%%%%%%%%%%%%%%%%%%%%%%%%%%%%%%%%%%%%%%%%%%%%%%%%%%%%%%%

\begin{lem} \label{lemg}
Let $m\ge 4$ be an integer for which there exists a regular ternary $m$-gonal form.
Then there exists a tight regular ternary complete quadratic polynomial
$$
g=g(x_1,x_2,x_3)=a_1(cx_1+\alpha_1)^2+a_2(cx_2+\alpha_2)^2+a_3(cx_3+\alpha_3)^2
$$
such that the following conditions hold:
\begin{enumerate} [(i)]
\item The positive integers $a_1,a_2,a_3$ are relatively prime, and $a_1\le a_2\le a_3$.
\item The $\z$-lattice $\langle a_1,a_2,a_3\rangle$ is $p$-stable for every prime $p\nmid c$.
\item For every $i=1,2,3$, we have $\alpha_i\in \z$ and $0<\alpha_i<\dfrac{c}{2}$.
\item $(c,\alpha_1\alpha_2\alpha_3)=1$.
\end{enumerate}
\end{lem}

\begin{proof}
Let $a_1'P_m(x_1)+a_2'P_m(x_2)+a_3'P_m(x_3)$ be a primitive regular ternary $m$-gonal form.
Let
$$
L=\z \be_1+\z \be_2+\z \be_3 \simeq \langle a_1'c^2,a_2'c^2,a_3'c^2\rangle
$$
be a $\z$-lattice, and let $\bv=-\dfrac{d}{c}(\be_1+\be_2+\be_3)\in \q L$.
Then, by Proposition \ref{propcorres}, $L+\bv$ is a primitive tight regular $\z$-coset of conductor $c$.

By applying Lemma~\ref{lemlambda} finitely many times, if necessary, we obtain a primitive tight regular $\z$-coset $K+\bw$ of conductor $c$, where $K$ is a ternary diagonal $\z$-lattice that is $p$-stable for all primes $p\nmid c$.
There exists a $\z$-basis $\{ \mathbf{f}_1, \mathbf{f}_2, \mathbf{f}_3\}$ of $K$ such that
$$
K=\z \mathbf{f}_1+\z \mathbf{f}_2+\z \mathbf{f}_3 \simeq \langle a_1c^2,a_2c^2,a_3c^2\rangle,
$$
where $a_i$'s are positive integers satisfying $a_1\le a_2\le a_3$.
By Lemma~\ref{lemlambda} and the definition of $\lambda_q$-transformations, we have $\mathfrak{s}(L)=\mathfrak{s}(K)$, and hence $(a_1,a_2,a_3)=1$.
Note that
$$
\mathbf{w}=\dfrac{w_1}{c}\mathbf{f}_1+\dfrac{w_2}{c}\mathbf{f}_2+\dfrac{w_3}{c}\mathbf{f}_3
$$
for some integers $w_1,w_2,w_3$ with $(w_1w_2w_3,c)=1$.
For $i=1,2,3$, write $w_i=cq_i+\alpha_i$ with $-\dfrac{c}{2}<\alpha_i<\dfrac{c}{2}$.
Then
$$
K+\bw=\z \mathbf{f}_1+\z \mathbf{f}_2+\z \mathbf{f}_3+\bw=\z \mathbf{f}_1+\z \mathbf{f}_2+\z \mathbf{f}_3+\left( \dfrac{\alpha_1}{c}\mathbf{f}_1+\dfrac{\alpha_2}{c}\mathbf{f}_2+\dfrac{\alpha_3}{c}\mathbf{f}_3\right).
$$
Define a complete quadratic polynomial $g=g(x_1,x_2,x_3)$ by
\begin{align*}
g(x_1,x_2,x_3)&=Q(x_1\mathbf{f}_1+x_2\mathbf{f}_2+x_3\mathbf{f}_3+\bw)\\
&=a_1(cx_1+\alpha_1)^2+a_2(cx_2+\alpha_2)^2+a_3(cx_3+\alpha_3)^2.
\end{align*}
Since $(cx_i+\alpha_i)^2=(c(-x_i)-\alpha_i)^2$, the eight complete quadratic polynomials
$$
a_1(cx_1\pm \alpha_1)^2+a_2(cx_2\pm \alpha_2)^2+a_3(cx_3\pm \alpha_3)^2
$$
are all equivalent.
Replacing $g$ by one of the eight equivalent polynomials, we may assume that $0<\alpha_i<\dfrac{c}{2}$ for each $i$.
For every prime $p\nmid c$, the ternary diagonal $\z$-lattice $\langle a_1,a_2,a_3\rangle$ is isometric to $K$ over $\z_p$, and hence is $p$-stable.
Finally, for each $i=1,2,3$, we have $(\alpha_i,c)=(w_i,c)=1$.
This completes the proof.
\end{proof}

%%%%%%%%%%%%%%%%%%%%%%%%%%%%%%%%%%%%%%%%%%%%%%%%%%%%%%%%%%%%%%%%%%%%%%%%%%%%%%%%%%%%%%%%%%

\begin{rmk} \label{rmkdelta}
With the notation of Lemma~\ref{lemg}, one has $c-2\alpha_i\ge \delta$ for $i=1,2,3$.
Indeed, if $m\equiv 1\pmod{2}$, then $c\equiv 2\alpha_i \equiv 2\pmod{4}$ and $c>2\alpha_i$, so $c-2\alpha_i \ge 4$.
If $m\equiv 2\pmod{4}$, then $c\equiv 0\pmod{4}$ while $2\alpha_i \equiv 2\pmod{4}$, and hence $c-2\alpha_i \ge 2$.
If $m\equiv 0\pmod{4}$, the inequality is immediate since $c-2\alpha_i \ge 1=\delta$.
\end{rmk}

%%%%%%%%%%%%%%%%%%%%%%%%%%%%%%%%%%%%%%%%%%%%%%%%%%%%%%%%%%%%%%%%%%%%%%%%%%%%%%%%%%%%%%%%%%

From now on, we assume that $m$ is an integer greater than or equal to $4$.
The following three lemmas will be used repeatedly in the proof of Theorem~\ref{thmmain}.
For the reader's convenience, we first fix the notation.

Throughout Lemmas~\ref{lempdividec}--\ref{lembound}, we assume that the ternary complete quadratic polynomial
$$
g=g(x_1,x_2,x_3)=a_1(cx_1+\alpha_1)^2+a_2(cx_2+\alpha_2)^2+a_3(cx_3+\alpha_3)^2
$$
is tight regular, where $a_i$'s and $\alpha_i$'s satisfy conditions~(i)--(iv) of Lemma~\ref{lemg}.
Let $J$ be a $\z$-lattice with Gram matrix $\langle a_1,a_2,a_3\rangle$.
Define $T$ to be the set of all primes $p\ge 5$ such that $J_p$ is anisotropic.
Then every element of $T$ divides $a_1a_2a_3$, and hence $T$ is finite.
We write
$$
T=\{p_1<p_2<\cdots<p_t\}
$$
so that $t=|T|$.
Let $\kappa$ be a positive integer coprime to all primes in $T$, and let $\nu$ be a nonnegative integer such that $\delta c(\kappa n+\nu)+\tail$ is represented by $g$ over $\z_2$ and $\z_3$ for every integer $n$.

%%%%%%%%%%%%%%%%%%%%%%%%%%%%%%%%%%%%%%%%%%%%%%%%%%%%%%%%%%%%%%%%%%%%%%%%%%%%%%%%%%%%%%%%%%

\begin{lem} \label{lempdividec}
Let $\beta(n)=\delta cn+\tail$ for $n\in \n_0$.
Then the following hold:
\begin{enumerate}[(i)]
\item The minimum of $g$ is equal to $\tail$.
\item For every prime divisor $p$ of $c$, the integer $\beta(n)$ is represented by $g$ over $\z_p$.
\item For a prime $p$ not dividing $c$, the integer $\beta(n)$ is represented by $g$ over $\z_p$ if and only if it is represented by the $\z$-lattice $\langle a_1,a_2,a_3\rangle$ over $\z_p$.
\end{enumerate}
\end{lem}

\begin{proof}
(i) This is immediate.
(ii) Let $p$ be a prime dividing $c$.
Without loss of generality, we may assume that $(a_1,p)=1$.
 Local Square Theorem \cite[63:1]{OM}, we have
$$
a_1^2\alpha_1^2 + a_1\delta cn \in (\z_p^{\times})^2.
$$
Note that if $p=2$, then $c$ is even, and hence $8\mid \delta c$.
Therefore, the quadratic equation
$$
a_1cx^2 + 2a_1\alpha_1x - \delta n = 0
$$
has a solution, say $x_1$, in $\z_p$.
This implies that
$$
a_1(cx_1+\alpha_1)^2 = \delta cn + a_1\alpha_1^2,
$$
and hence
$$
\delta cn+\tail
= a_1(cx_1+\alpha_1)^2 + a_2\alpha_2^2 + a_3\alpha_3^2=g(x_1,0,0).
$$
(iii) This follows immediately since $c\in \z_p^{\times}$.
\end{proof}

%%%%%%%%%%%%%%%%%%%%%%%%%%%%%%%%%%%%%%%%%%%%%%%%%%%%%%%%%%%%%%%%%%%%%%%%%%%%%%%%%%%%%%%%%%

\begin{lem} \label{lemlocal}
Let $\beta(u)=\delta c(\kappa u+\nu)+\tail$ for $u\in \n_0$.
For any positive integer $s$ with $s\ge t$ and any positive integer $n$, we have
$$
\vert\{0\le u\le n-1 : \beta(u)\ra g\}\vert \ge \eta(n,s).
$$
\end{lem}

\begin{proof}
We first claim that, for every $u\in \n_0$, the integer $\beta(u)$ is represented by $g$ over $\z_p$ for all $p\in \Omega\smallsetminus T$.
Clearly, $\beta(u)$ is represented by $g$ over $\mathbb{R}$.
By the choice of $\nu$, the integer $\beta(u)$ is represented by $g$ over $\z_p$ for $p=2,3$.
Let $p\ge 5$ be a prime not contained in $T$.
If $p$ divides $c$, then $\beta(u)$ is represented by $g$ over $\z_p$ by Lemma~\ref{lempdividec}(ii).
Thus, we may assume that $p\nmid c$.
Then $J$ is $p$-stable and isotropic over $\z_p$, so that $Q(J_p)=\z_p$.
Hence $\beta(u)$ is represented by $J$ over $\z_p$.
By Lemma~\ref{lempdividec}(iii), it follows that $\beta(u)$ is represented by $g$ over $\z_p$.
This proves the claim.

One sees that
\begin{align*}
&\vert\{0\le u\le n-1 : \beta(u) \ra g\}\vert \\
&=\vert\{0\le u\le n-1 : \beta(u) \ra g\ \text{over}\ \z_p\ \text{for all}\ p\in \Omega \}\vert \\
&=\vert\{0\le u\le n-1 : \beta(u) \ra g\ \text{over}\ \z_p\ \text{for all}\ p\in T\}\vert \\
&=n- \vert\{0\le u\le n-1 : \beta(u) \nra g\ \text{over}\ \z_p\ \text{for some}\ p\in T\}\vert \\
&=n- \vert\{0\le u\le n-1 : \beta(u) \nra J \ \text{over}\ \z_p\ \text{for some}\ p\in T\}\vert,
\end{align*}
where the first equality follows from the tight regularity of $g$, the second from the above claim, the third is immediate, and the last follows from Lemma~\ref{lempdividec}(iii).
Finally, we estimate
\begin{align*}
&n- \vert\{0\le u\le n-1 : \beta(u) \nra J \ \text{over}\ \z_p\ \text{for some}\ p\in T\}\vert \\
&\ge n-\sum_{p\in T}\vert\{0\le u<n-1 : \beta(u) \nra J \ \text{over}\ \z_p\}\vert \\
&\ge n-\sum_{p\in T} \psi_p(n)\ge \eta(n,t)\ge \eta(n,s),
\end{align*}
where the second inequality follows from Equation \eqref{eqpsi}, and the third from the definition of $\eta(n,t)$.
This completes the proof.
\end{proof}

%%%%%%%%%%%%%%%%%%%%%%%%%%%%%%%%%%%%%%%%%%%%%%%%%%%%%%%%%%%%%%%%%%%%%%%%%%%%%%%%%%%%%%%%%%

\begin{lem} \label{lembound}
Let $s$ be a positive integer with $s\ge t$, and let $n$ and $N$ be positive integers.
\begin{enumerate}[(i)]
\item If 
$$
\left| \{ 0\le u\le n : \delta c(\kappa u+\nu)+\tail \ra g \} \right| \ge 2,
$$
then $a_1\le \kappa n+\nu$.

\item If $\eta(n,s)>8N^3$, then
$$
a_1(N^2c+2N)\le \delta(\kappa (n-1)+\nu).
$$
In particular, if $\eta(n,s)=9$, then
$$
a_1\le \left\lfloor \dfrac{\delta(\kappa (n-1)+\nu)}{c+2}\right\rfloor,
$$
where $\lfloor \cdot \rfloor$ denotes the floor function.

\item If $\eta(n,s)>2N$ and $N^2c+2N>\delta (\kappa (n-1)+\nu)$, then
$$
a_2\le \kappa (n-1)+\nu.
$$
\end{enumerate}
\end{lem}

\begin{proof}
\noindent (i)
By assumption, there exists an integer $n_0$ with $0\le n_0\le n$ such that
$$
\delta c(\kappa n_0+\nu)+\tail=a_1y_1^2+a_2y_2^2+a_3y_3^2,
$$
where the $y_i$ are integers satisfying $y_i\equiv \pm \alpha_i\pmod c$ and there exists an index $j\in \{1,2,3\}$ such that $y_j\ge c-\alpha_j$.
Thus,
\begin{align*}
a_j(c-\alpha_j)^2+\sum_{k\in \{1,2,3\}\smallsetminus \{j\}}a_k\alpha_k^2&\le a_1y_1^2+a_2y_2^2+a_3y_3^2\\
&\le \delta c(\kappa n+\nu)+\tail.
\end{align*}
This implies that
$$
a_j(c^2-2c\alpha_j)\le \delta c(\kappa n+\nu).
$$
By Remark \ref{rmkdelta}, we have $a_j\le \kappa n+\nu$.
Since $a_1\le a_2\le a_3$, it follows that $a_1\le \kappa n+\nu$.

\noindent (ii)
By Lemma \ref{lemlocal}, we have
$$
\vert \{ 0\le u\le n-1 : \delta c(\kappa u+\nu)+\tail \ra g\} \vert \ge \eta(n,s)>8N^3.
$$
On the other hand,
$$
\vert \{a_1y_1^2+a_2y_2^2+a_3y_3^2 : y_i\in \{ \alpha_i,c-\alpha_i,c+\alpha_i,\dots,Nc-\alpha_i\} \ \text{for}\ i=1,2,3\} \vert \le 8N^3.
$$
It follows that there exists an integer $n_0$ with $0\le n_0\le n-1$ such that
$$
\delta c(\kappa n_0+\nu)+\tail=a_1y_1^2+a_2y_2^2+a_3y_3^2,
$$
where the $y_i$ are integers with $y_i\equiv \pm \alpha_i\pmod c$ and there exists an index $j\in \{1,2,3\}$ such that $y_j\ge Nc+\alpha_j$.
Thus,
\begin{align*}
a_j(Nc+\alpha_j)^2+\sum_{k\in \{1,2,3\}\smallsetminus \{j\}}a_k\alpha_k^2&\le a_1y_1^2+a_2y_2^2+a_3y_3^2\\
&\le \delta c(\kappa (n-1)+\nu)+\tail.
\end{align*}
Hence,
$$
a_j(N^2c^2+2Nc\alpha_j)\le \delta c(\kappa (n-1)+\nu).
$$
Since $a_1\le a_2\le a_3$ and $\alpha_j\ge 1$, we obtain
$$
a_1(N^2c+2N)\le \delta (\kappa (n-1)+\nu).
$$

\noindent (iii)
Since $\eta(n,s)>2N$, it follows from Lemma~\ref{lemlocal} that
$$
\vert \{ 0\le u\le n-1 : \delta c(\kappa u+\nu)+\tail \ra g\} \vert >2N.
$$
Since $a_1,\alpha_1\in \n$, the inequality
$$
N^2c+2N>\delta (\kappa(n-1)+\nu)
$$
implies that
$$
a_1(Nc+\alpha_1)^2+a_2\alpha_2^2+a_3\alpha_3^2
> \delta c (\kappa (n-1)+\nu)+\tail.
$$
On the other hand,
$$
\vert \{ a_1y_1^2+a_2\alpha_2^2+a_3\alpha_3^2 : y_1\in \{ \alpha_1,c-\alpha_1,c+\alpha_1,\dots,Nc-\alpha_1\} \vert=2N.
$$
It follows that there exists an integer $n_0$ with $0\le n_0\le n-1$ such that
$$
\delta c(\kappa n_0+\nu)+\tail=a_1y_1^2+a_2y_2^2+a_3y_3^2,
$$
where the $y_i$ are integers satisfying $y_i\equiv \alpha_i\pmod c$ and there exists an index $j\in \{2,3\}$ such that $y_j\ge c-\alpha_j$.
Thus,
\begin{align*}
a_j(c-\alpha_j)^2+\sum_{k\in \{1,2,3\} \smallsetminus \{j\}} a_k\alpha_k^2&\le  a_1y_1^2+a_2y_2^2+a_3y_3^2\\
&\le \delta c(\kappa (n-1)+\nu)+\tail.
\end{align*}
This implies that
$$
a_j(c-2\alpha_j)\le \delta (\kappa (n-1)+\nu).
$$
By the choice of $j$ and Remark \ref{rmkdelta}, it follows that
$$
a_2\le a_j\le \kappa (n-1)+\nu.
$$
This completes the proof.
\end{proof}

%%%%%%%%%%%%%%%%%%%%%%%%%%%%%%%%%%%%%%%%%%%%%%%%%%%%%%%%%%%%%%%%%%%%%%%%%%%%%%%%%%%%%%%%%%

We are now ready to prove the main theorem.

%%%%%%%%%%%%%%%%%%%%%%%%%%%%%%%%%%%%%%%%%%%%%%%%%%%%%%%%%%%%%%%%%%%%%%%%%%%%%%%%%%%%%%%%%%

\begin{proof}[(Proof of Theorem \ref{thmmain})]
Let $m\ge 4$ be an integer for which there exists a regular ternary $m$-gonal form.
By Lemma~\ref{lemg}, there exists a tight regular ternary complete quadratic polynomial
$$
g=g(x_1,x_2,x_3)=a_1(cx_1+\alpha_1)^2+a_2(cx_2+\alpha_2)^2+a_3(cx_3+\alpha_3)^2
$$
satisfying conditions (i)--(iv) of that lemma.
Let $T$ be the set of all primes $p\ge 5$ with $(p,c)=1$ such that the $\z$-lattice $\langle a_1,a_2,a_3\rangle$ is anisotropic over $\z_p$.
Then every prime in $T$ divides $a_1a_2a_3$, and hence $T$ is finite.
Write $T=\{p_1<p_2<\cdots<p_t\}$, so that $|T|=t$.
Later, we divide the argument into four cases, in each of which we show that 
$t\le 4,6,6$, or $9$, respectively.
Since these bounds suffice, we may assume that $t>0$.
We next derive an upper bound for $c$, which in turn yields an upper bound for $m$.
Henceforth, we assume that $c\ge 36$.
By Lemma~\ref{lempdividec}(ii),
$$
\delta cn+\tail \ra g \quad \text{over } \z_p
$$
for every prime divisor $p$ of $c$ and every $n\in \n_0$.
Note that $2\mid c$ if and only if $m\not\equiv 0 \pmod{4}$, and $3\mid c$ if and only if $m\equiv 2 \pmod{3}$.
Using Lemma~\ref{lem1}, choose an integer $v$ satisfying the conditions of the lemma for $p=p_1$ and $u=\delta c$.
Define
$$
\kappa=\begin{cases}
1 & \text{if } m\not\equiv 0 \pmod{4},\ m\equiv 2 \pmod{3},\\
3 & \text{if } m\not\equiv 0 \pmod{4},\ m\not\equiv 2 \pmod{3},\\
4 & \text{if } m\equiv 0 \pmod{4},\ m\equiv 2 \pmod{3},\\
12 & \text{if } m\equiv 0 \pmod{4},\ m\not\equiv 2 \pmod{3}.
\end{cases}
$$
Define an integer $\nu_1$ as follows.
If $m\not\equiv 0 \pmod{4}$ and $m\equiv 2 \pmod{3}$, set $\nu_1=0$.
If $m\not\equiv 0 \pmod{4}$ and $m\not\equiv 2 \pmod{3}$, choose $\nu_1 \in \{0,1\}$ such that
$$
\delta cp_1^2\nu_1+\delta cv+\tail \not\equiv 0 \pmod{3}.
$$
If $m\equiv 0 \pmod{4}$ and $m\equiv 2 \pmod{3}$, then by Lemma~\ref{lemz23} we may choose $\nu_1 \in \{0,1,2\}$ such that
$$
\delta cp_1^2(4n+\nu_1)+\delta cv+\tail
$$
is represented by $g$ over $\z_2$ for every integer $n$.
If $m\equiv 0 \pmod{4}$ and $m\not\equiv 2 \pmod{3}$, then by Lemma~\ref{lemz23} we may choose $\nu_1 \in \{0,1,2,3,4\}$ such that
$$
\delta cp_1^2(12n+\nu_1)+\delta cv+\tail
$$
is represented by $g$ over $\z_2$ and $\z_3$ for every integer $n$.

By Lemma~\ref{lem2}, choose an integer $w$ with $0\le w<(t+3)2^{t-3}$ such that
$$
(\delta cp_1^2(\kappa w+\nu_1)+\delta cv+\tail,\; p_2p_3\cdots p_t)=1.
$$
Set $k=p_1^2(\kappa w+\nu_1)+v$.
Then
$$
\delta ck+\tail \ra g \quad \text{over } \z_p
$$
for all primes $p$.
Since $g$ is tight regular and
$$
\delta ck+\tail \ge \min(g),
$$
it follows that $\delta ck+\tail$ is represented by $g$.
Hence there exist integers $y_1,y_2,y_3$ such that
$$
\delta ck+\tail=a_1(cy_1+\alpha_1)^2+a_2(cy_2+\alpha_2)^2+a_3(cy_3+\alpha_3)^2.
$$
From the choice of $v$ and $k$, it follows that
$$
\delta ck+a_1\alpha_1^2+a_2\alpha_2^2
=(\delta cv+a_1\alpha_1^2+a_2\alpha_2^2)\epsilon_{p_1}^2
$$
for some $\epsilon_{p_1} \in \z_{p_1}^{\times}$.
Therefore,
$$
\delta ck+a_1\alpha_1^2+a_2\alpha_2^2 \nra \langle a_1,a_2\rangle 
\quad \text{over } \z_{p_1}.
$$
This implies that $(cy_3+\alpha_3)^2 \ne \alpha_3^2$.
Observe that
$$
\{(cx+\alpha_3)^2 : x\in \z \}
=\{ \alpha_3^2,\ (c-\alpha_3)^2,\ (c+\alpha_3)^2,\ (2c-\alpha_3)^2,\ \dots \},
$$
which is a strictly increasing sequence.
Hence $(cy_3+\alpha_3)^2 \ge (c-\alpha_3)^2$, and therefore
$$
a_1\alpha_1^2+a_2\alpha_2^2+a_3(c-\alpha_3)^2
\le \delta ck+\tail.
$$
It follows that
$$
a_3 \le \frac{\delta k}{c-2\alpha_3}.
$$
Since $c-2\alpha_3 \ge \delta$ by Remark~\ref{rmkdelta}, we conclude that $a_3 \le k$.

Now we define a nonnegative integer $\nu_2$ such that
$$
\delta c(\kappa u+\nu_2)+\tail
$$
is represented by $g$ over $\z_2$ and $\z_3$ for every $u\in \n_0$.
If $m\not\equiv 0 \pmod{4}$ and $m\equiv 2 \pmod{3}$, set $\nu_2=0$.
If $m\not\equiv 0 \pmod{4}$ and $m\not\equiv 2 \pmod{3}$, choose $\nu_2\in \{0,1\}$ such that
$$
\delta c\nu_2+\tail \not\equiv 0 \pmod{3}.
$$
If $m\equiv 0 \pmod{4}$ and $m\equiv 2 \pmod{3}$, then by Lemma~\ref{lemz23} we may choose $\nu_2\in \{0,1,2\}$ such that
$$
\delta c(\kappa u+\nu_2)+\tail \ra g \quad \text{over } \z_2
$$
for all $u\in \n_0$.
If $m\equiv 0 \pmod{4}$ and $m\not\equiv 2 \pmod{3}$, then by Lemma~\ref{lemz23} we may choose $\nu_2\in \{0,1,2,3,4\}$ such that
$$
\delta c(\kappa u+\nu_2)+\tail \ra g \quad \text{over } \z_2 \text{ and } \z_3
$$
for all $u\in \n_0$.

We divide the argument into the following four cases.
Since the arguments are similar, we treat the first case in detail and only sketch the proofs of the remaining cases.

\noindent \textbf{(Case 1)} Suppose that $m\not\equiv 0 \pmod{4}$ and $m\equiv 2 \pmod{3}$.
In this case,
$$
\delta \in \{2,4\}, \quad \kappa=1, \quad \nu_2=0.
$$

\noindent \textbf{Step 1.} We show that $t\le 15$.
Since
$$
a_3\le k\le p_1^2(t+3)2^{t-3},
$$
it follows that
$$
p_1p_2\cdots p_t \le a_1a_2a_3 \le a_3^3 \le p_1^6((t+3)2^{t-3})^3.
$$
Suppose that $t\ge 16$.
Then, by Proposition~\ref{propproduct}(i),
$$
r_7r_8\cdots r_t>((t+3)2^{t-3})^3,
$$
and hence
$$
p_1p_2\cdots p_t \ge p_1^6 r_7r_8\cdots r_t > p_1^6((t+3)2^{t-3})^3,
$$
which is a contradiction.
Therefore, $t\le 15$.

\noindent \textbf{Step 2.} We show that $t\le 6$.
Note that $\eta(49,15)=5$.
By Lemma~\ref{lemlocal},
$$
\left| \{ 0\le n\le 48 : \delta cn+\tail \ra g\} \right| \ge 5,
$$
and hence
$$
\left| \{ 0\le n\le 45 : \delta cn+\tail \ra g\} \right| \ge 2.
$$
By Lemma~\ref{lembound}(i), it follows that $a_1\le 45$.

Note that $\eta(50,15)=7$, and that $9c+6>4\cdot 49$ since $c\ge 36$.
Applying Lemma~\ref{lembound}(iii) with $N=3$, we obtain $a_2\le 49$.

Thus,
$$
p_1p_2\cdots p_t \le a_1a_2a_3 \le 45\cdot 49\cdot p_1^2(t+3)2^{t-3}.
$$
Suppose that $t\ge 7$.
Then, by Proposition~\ref{propproduct}(ii),
$$
r_3r_4\cdots r_t > 45\cdot 49\cdot (t+3)2^{t-3},
$$
and hence
$$
p_1p_2\cdots p_t > 45\cdot 49\cdot p_1^2(t+3)2^{t-3},
$$
which is a contradiction.
Therefore, $t\le 6$.

\noindent \textbf{Step 3.} We show that $t\le 4$.
Note that $\eta(25,6)=9$.
By Lemma~\ref{lembound}(ii) with $n=25$ and $s=6$, we have
$$
a_1 \le \left\lfloor \frac{\delta(25-1)}{c+2} \right\rfloor
\le \left\lfloor \frac{96}{38} \right\rfloor = 2.
$$

Note that $\eta(19,6)=5$.
Since $c\ge 36$, we have $4c+4>4\cdot 18$.
Hence, applying Lemma~\ref{lembound}(iii) with $N=2$, we obtain $a_2\le 18$.

Thus,
$$
p_1p_2\cdots p_t \le a_1a_2a_3 \le 2\cdot 18\cdot p_1^2(t+3)2^{t-3}.
$$
By Proposition~\ref{propproduct}(iii), it follows that $t\le 4$.

\noindent \textbf{Step 4.} We derive an upper bound for $c$.
Note that $\eta(20,4)=9$.
By Lemma~\ref{lembound}(ii) with $N=1$, we have $a_1(c+2)\le 19\delta$.
Therefore,
$$
c \le \frac{19\delta}{a_1}-2 \le 19\delta - 2
=
\begin{cases}
74 & \text{if } m\equiv 1 \pmod{2},\ m\equiv 2 \pmod{3},\\
36 & \text{if } m\equiv 2 \pmod{4},\ m\equiv 2 \pmod{3}.
\end{cases}
$$
If $m\equiv 5 \pmod{6}$, then $c=2(m-2)$, and hence $m\le 39$.
Using the congruence $m\equiv 5 \pmod{6}$, we further obtain $m\le 35$.
Similarly, one deduces that $m\le 38$ when $m\equiv 2 \pmod{12}$.
This completes the proof of Case~1.

\vskip 7pt

\noindent \textbf{(Case 2)} Suppose that $m\not\equiv 0 \pmod{4}$ and $m\not\equiv 2 \pmod{3}$.
In this case,
$$
\delta \in \{2,4\}, \quad \kappa=3, \quad 0\le \nu_2\le 1.
$$

\noindent \textbf{Step 1.} We show that $t\le 17$.
Since
$$
a_3\le k\le p_1^2(3(t+3)2^{t-3}-1),
$$
it follows that
$$
p_1p_2\cdots p_t \le a_1a_2a_3 \le a_3^3 \le p_1^6(3(t+3)2^{t-3}-1)^3.
$$
By Proposition~\ref{propproduct}(iv), it follows that $t\le 17$.

\noindent \textbf{Step 2.} We show that $t\le 8$.
Note that $\eta(49,17)=3$.
By Lemma~\ref{lemlocal},
$$
\# \{ 0\le n\le 48 : \delta c(3n+\nu_2)+\tail \ra g\} \ge 3,
$$
and hence
$$
\# \{ 0\le n\le 47 : \delta c(3n+\nu_2)+\tail \ra g\} \ge 2.
$$
By Lemma~\ref{lembound}(i), it follows that $a_1\le 142$.

Note that $\eta(102,17)=17$, and that $64c+16>4(3\cdot 101+1)$ since $c\ge 36$.
Applying Lemma~\ref{lembound}(iii) with $N=8$, we obtain $a_2\le 304$.

Thus,
$$
p_1p_2\cdots p_t \le a_1a_2a_3 \le 142\cdot 304\cdot p_1^2(3(t+3)2^{t-3}-1).
$$
By Proposition~\ref{propproduct}(v), it follows that $t\le 8$.

\noindent \textbf{Step 3.} We show that $t\le 7$.
Note that $\eta(17,8)=2$.
By Lemma~\ref{lembound}(i), we obtain $a_1\le 3\cdot 16+1=49$.
Since $\eta(48,8)=13$, applying Lemma~\ref{lembound}(iii) with $N=6$, we obtain $a_2\le 142$.

Thus,
$$
p_1p_2\cdots p_t \le 49\cdot 142\cdot p_1^2(3(t+3)2^{t-3}-1).
$$
By Proposition~\ref{propproduct}(vi), it follows that $t\le 7$.

\noindent \textbf{Step 4.} We show that $t\le 6$.
Note that $\eta(30,7)=9$.
By Lemma~\ref{lembound}(ii), we obtain $a_1\le 9$.
Applying Lemma~\ref{lembound}(iii) with $N=4$ (and $\eta(30,7)=9$), we obtain $a_2\le 88$.

Thus,
$$
p_1p_2\cdots p_t \le a_1a_2a_3 \le 9\cdot 88\cdot p_1^2(3(t+3)2^{t-3}-1).
$$
By Proposition~\ref{propproduct}(vii), it follows that $t\le 6$.

\noindent \textbf{Step 5.} We derive an upper bound for $c$.
Note that $\eta(25,6)=9$.
By Lemma~\ref{lembound}(ii) with $N=1$, we have $a_1(c+2)\le 73\delta$.
Therefore,
$$
c \le 73\delta - 2
=
\begin{cases}
290 & \text{if } m\equiv 1 \pmod{2},\ m\not\equiv 2 \pmod{3},\\
144 & \text{if } m\equiv 2 \pmod{4},\ m\not\equiv 2 \pmod{3}.
\end{cases}
$$

\vskip 7pt

\vskip 7pt
\noindent \textbf{(Case 3)} Suppose that $m\equiv 0 \pmod{4}$ and $m\equiv 2 \pmod{3}$.
Then
$$
\delta =1, \quad \kappa=4, \quad 0\le \nu_2\le 2.
$$

The argument is parallel to that of Case~2, with $\kappa=4$ and $\delta=1$.

\noindent \textbf{Step 1.}
Since
$$
a_3 \le k \le p_1^2(4(t+3)2^{t-3}-1),
$$
it follows that
$$
p_1\cdots p_t \le a_1a_2a_3 \le a_3^3 \le p_1^6(4(t+3)2^{t-3}-1)^3,
$$
and hence $t\le 17$ by Proposition~\ref{propproduct}(viii).

\noindent \textbf{Step 2.}
Using Lemma~\ref{lemlocal} with $\eta(49,17)=3$, we obtain
$$
a_1 \le 190.
$$
Applying Lemma~\ref{lembound}(iii) with $\eta(74,17)=7$, we obtain
$$
a_2 \le 294.
$$
Thus, by Proposition~\ref{propproduct}(ix), we have $t\le 8$.

\noindent \textbf{Step 3.}
Similarly, using $\eta(17,8)=2$ and $\eta(28,8)=7$, we obtain
$$
a_1 \le 66, \quad a_2 \le 110,
$$
and hence $t\le 7$ by Proposition~\ref{propproduct}(x).

\noindent \textbf{Step 4.}
Using $\eta(30,7)=9$, Lemma~\ref{lembound}(ii) yields $a_1\le 3$.
Arguing as in the proof of Lemma~\ref{lembound}(iii), we obtain $a_2\le 90$.
Thus $t\le 6$ by Proposition~\ref{propproduct}(xi).

\noindent \textbf{Step 5.}
From Lemma~\ref{lembound}(ii) with $N=1$, we have $a_1(c+2)\le 98$,
and hence $c\le 96$.

\vskip 7pt

\noindent \textbf{(Case 4)} Suppose that $m\equiv 0 \pmod{4}$ and $m\not\equiv 2 \pmod{3}$.
Then
$$
\delta =1, \quad \kappa=12, \quad 0\le \nu_2\le 4.
$$

The initial steps follow the same pattern as in Case~2.

\noindent \textbf{Step 1.}
Since
$$
a_3 \le k \le p_1^2(12(t+3)2^{t-3}-7),
$$
we obtain
$$
p_1\cdots p_t \le a_3^3 \le p_1^6(12(t+3)2^{t-3}-7)^3,
$$
and hence $t\le 19$ by Proposition~\ref{propproduct}(xii).

\noindent \textbf{Step 2.}
Using Lemma~\ref{lemlocal} with $\eta(50,19)=3$, we obtain
$$
a_1 \le 580.
$$
Applying Lemma~\ref{lembound}(iii) with $\eta(125,19)=25$, we obtain
$$
a_2 \le 1492.
$$
Thus, by Proposition~\ref{propproduct}(xiii), we have $t\le 9$.

\noindent \textbf{Step 3.}
We derive an upper bound for $c$.
Assume, for contradiction, that $c\ge 356$.
Then, for $i=1,2$, we have
$$
a_i(4c-4\alpha_i)
= a_i(2c+2(c-2\alpha_i))
\ge a_i(2c+2)
> 12\cdot 59+4,
$$
and hence
$$
a_i(2c-\alpha_i)^2
+ \sum_{j\ne i} a_j\alpha_j^2
\ge c(12\cdot 59+\nu_2)+\tail.
$$

Since $\eta(60,9)=19$, if
$$
a_1\alpha_1^2+a_2\alpha_2^2+a_3(c+\alpha_3)^2
> c(12\cdot 59+\nu_2)+\tail,
$$
then
\begin{align*}
\{ 0&\le u\le 59 : c(12u+\nu_2)\ra g\} \\
&\subseteq \{ a_1y_1^2+a_2y_2^2+a_3y_3^2 :
y_3\in \{\alpha_3, c-\alpha_3\},\ 
y_i\in \{\alpha_i, c-\alpha_i, c+\alpha_i\} \ (i=1,2)\},
\end{align*}
which has cardinality at most $18$, a contradiction.

Therefore,
$$
a_1\alpha_1^2+a_2\alpha_2^2+a_3(c+\alpha_3)^2
\le c(12\cdot 59+\nu_2)+\tail,
$$
and hence
$$
a_3(c+2\alpha_3) \le 12\cdot 59+4.
$$

Since $c\ge 356$, it follows that $a_3=1$, and hence $a_1=a_2=a_3=1$.
Thus,
$$
c(12u+\nu_2)+\alpha_1^2+\alpha_2^2+\alpha_3^2 \ra g
\quad \text{for all } u\in \n_0.
$$
In particular,
$$
\left| \{ 0\le u\le 8 :
c(12u+\nu_2)+\alpha_1^2+\alpha_2^2+\alpha_3^2 \ra g\} \right| = 9.
$$

Arguing as in the proof of Lemma~\ref{lembound}(ii), we obtain
$$
(c+\alpha_j)^2 + \sum_{k\ne j} \alpha_k^2
\le c(12\cdot 8+\nu_2)+\alpha_1^2+\alpha_2^2+\alpha_3^2
$$
for some $j\in \{1,2,3\}$,
which is impossible since $c\ge 356$.
Therefore, $c\le 355$.
This completes the proof.
\end{proof}

%%%%%%%%%%%%%%%%%%%%%%%%%%%%%%%%%%%%%%%%%%%%%%%%%%%%%%%%%%%%%%%%%%%%%%%%%%%%%%%%%%%%%%%%%%

%%%%%%%%%%%%%%%%%%%%%%%%%%%%%%%%%%%%%%%%%%%%%%%%%%%%%%%%%%%%%%%%%%%%%%%%%%%%%

\end{document}